\documentclass[a4paper,11pt]{article}
\usepackage{amsmath,amsfonts,amssymb,amsthm}
\usepackage{graphicx,subfigure,booktabs,multirow}
\usepackage{mathrsfs,enumerate,color,bm}
\usepackage{algorithm,algpseudocode}
\usepackage[textwidth=14cm]{geometry}
\usepackage{color}
\usepackage{hyperref}
\usepackage{epsfig}
\usepackage{graphicx}
\usepackage{epstopdf}

\newtheorem{theorem}{Theorem}[section]
\newtheorem{lemma}[theorem]{Lemma}
\newtheorem{proposition}[theorem]{Proposition}

\newcommand{\rank}{\operatorname{rank}}

\newcommand{\re}{\operatorname{Re}}

\graphicspath{}
\title{Fast Alternating Projections on Manifolds Based on Tangent Spaces}
\author{Guang-Jing Song\thanks{School of Mathematics and Information Sciences, Weifang University,
Weifang 261061, P.R. China. (email: sgjshu@163.com)}
\and
Michael K. Ng\thanks{Department of Mathematics, The University of Hong Kong, Pokfulam, Hong Kong
(email: mng@maths.hku.hk).
M. Ng's research supported in part by the HKRGC GRF 12306616,
12200317, 12300218 and 12300519, and HKU 104005583.
}
}
\begin{document}
\maketitle

\begin{abstract}
In this paper, we study alternating projections
on nontangential manifolds based on the tangent spaces.
The main motivation is that the projection of a point
onto a manifold can be computational expensive. We propose
to use the tangent space of the point in the manifold
to approximate the projection onto the manifold in order to reduce the computational cost.
We show that the sequence generated by alternating projections on two nontangential manifolds
based on tangent spaces,
converges linearly to a point in the intersection of the two manifolds where the convergent point
is close to the optimal solution. Numerical examples for nonnegative low rank matrix
approximation and low rank image quaternion matrix (color image) approximation, are given to demonstrate that
the performance of the proposed method is better than that of the classical alternating projection method
in terms of computational time.
\end{abstract}


\noindent
AMS subject classiflcations. 15A23, 65f22.

\section{Introduction}

Throughout this paper, let ${\cal K}$  be a finite dimensional Hilbert space over $\mathbb{R}$, $\mathcal{M}_{1}$ and $\mathcal{M}_{2}$  be two manifolds included in ${\cal K}$.
The corresponding projection operators on $\mathcal{M}_{1}$, $\mathcal{M}_{2}$ and $\mathcal{M}=\mathcal{M}_{1}\cap \mathcal{M}_{2}$
are denoted by $\pi_{1},\pi_{2}$ and $\pi$ respectively. In this paper, we are interested to determine
a solution defined by $\pi$ onto ${\cal M}$.
For example, nonnegative rank $r$ matrix approximation for nonnegative matrices aims to find a nonnegative rank $r$ matrix such that the
distance between such matrix and the given nonnegative matrix is as small as possible.
Here ${\cal M}_1$ refers to the set of rank $r$ real matrices, ${\cal M}_2$ refers to the set of matrices with nonnegative entries,
and the projection refers to the closest matrix to the given nonnegative matrix $A$, i.e.,
\begin{equation}\label{pj1}
\pi(A) = \underset{X \in {\cal M}}{\operatorname{argmax}} \| A - X \|_F^2, \quad
\pi_1(A) = \underset{X \in {\cal M}_1}{\operatorname{argmax}} \| A - X \|_F^2 \quad
\pi_2(A) = \underset{X \in {\cal M}_2}{\operatorname{argmax}} \| A - X \|_F^2.
\end{equation}

The classical alternating projection method is to determine a solution by using two projections $\pi_1$ onto ${\cal M}_1$
and $\pi_2$ onto ${\cal M}_2$ iteratively. The method is widely used in many fields, for instance
in signal processing \cite{combettes1993signal}, finance \cite{higham2002computing}, machine learning \cite{widrow1987adaptive},
numerical linear algebra \cite{chenchu1996}, image processing \cite{grigoriadis1996,Griogoriadis2000,orsi2006},
and other applications
(see \cite{grigoriadis1994application,hamaker1978angles,kayalar1988error,lee1997conformal,levi1983signal} and references therein).
In the literature, Schwarz \cite{schwarz1986} firstly studied the alternating projection method.
When $\mathcal{M}_{1}$ and $\mathcal{M}_{2}$ are affine linear manifolds,
von Neumann \cite{Neuman1950} proved that the sequence derived by the alternating projection method is globally convergent to
a solution given by the projection onto ${\cal M} = {\cal M}_1 \cap {\cal M}_2$ under the assumption that ${\cal M}
\neq \emptyset$.
And the convergence rate is shown to be linear and governed by
the angle between $\mathcal{M}_{1}$ and $\mathcal{M}_{2}$.
However, when $\mathcal{M}_{1}$ and $\mathcal{M}_{2}$ are nonlinear manifolds, the corresponding results cannot be derived in general, i.e.,
even if $\mathcal{M}_{1}\cap \mathcal{M}_{2}\neq \emptyset$, the sequence generated by the alternating projection method
is not necessary to be convergent.
 Lewis and Malick \cite{lewis2008alternating}
studied the alternating projection method on two smooth manifolds which can be approximated by some affine subspaces.
They proved that if the two manifolds intersect transversally, the sequence can be excepted to be convergent to a point
in the intersection of the two manifolds with a linear rate.
Recently, Andersson and Carlsson \cite{andersson2013alternating} showed that
if the two manifolds have ``nontangential'' intersection points, the
sequence of alternating projections converges linearly to a point in the intersection which is sufficiently close to the optimal solution.

In the alternating projection method, the optimal solution of each projection is assumed. In general,
the computational cost of each projection can be be expensive.
The main aim of this paper is to propose
to use the tangent space of the point in the manifold
to approximate the projection onto the manifold in order to reduce the computational cost.
We show that the sequence generated by alternating projections on two nontangential manifolds
based on tangent spaces,
converges linearly to a point in the intersection of the two manifolds where the convergent point
is close to the optimal solution.
As an application, we demonstrate the proposed algorithm to solve nonnegative low rank matrix
approximation to nonnegative matrices. We also present numerical examples
for nonnegative low rank matrix approximation and
low rank color image approximation, and demonstrate
that the performance of the proposed method is better than that of
other testing methods in terms of computational time.

The rest of this paper is organized as follows.
In Section \ref{sec:main1}, we present the proposed tangent space-based alternating projection
method. Nonnegative low rank matrix approximation problem is used as an example for illustration.
In Section 3, we show the convergence of the proposed tangent space-based
alternating projection method.
In Section 4, numerical examples are given to show the advantages of the proposed method.
Finally, some concluding remarks are given in Section \ref{sec:clu}.

\section{The Proposed Projection Method}
\label{sec:main1}

\subsection{Preliminaries}

In this subsection,
we first provide a review of some necessary concepts and preliminaries of the differential geometry (for details, we refer to \cite{berger1988}).  Let $\mathcal{K}$ be a Euclidean space, i.e., a Hilbert space of finite dimension $n\in \emph{N}$.
Given $A\in \mathcal{K}$ and $r>0$, we write $\mathcal{B}(A,r)$ for the open ball centered at $A$
with radius $r.$ Any subset $\mathcal{M}$ of $\mathcal{K}$ will be given the induced topology
from $\mathcal{K}$. Let $p\geq 1$, and let $\mathcal{M}\subseteq \mathcal{K}$ be an $m$-dimensional
$\mathbb{C}^p$-manifold. We recall that around each $A \in \mathcal{M}$,
there exist an injective $\mathbb{C}^{p}$-immersion $\phi$ on an open set $U$ in
$\mathbb{R}^{m}$ such that
\begin{align}\label{aff1}
\mathcal{M} \cap \mathcal{B}(A,s)=\textit{Im}~ \phi \cap \mathcal{B}(A,s)
\end{align}
for some $s>0$, where $\textit{Im}(\phi)$ denotes the image of $\phi$.
If $A=\phi(x_{A})$, we define the tangent space
$T_{\mathcal{M}}(A)$
by
$$
T_{\mathcal{M}}(A)=\emph{Range} \ d\phi(x_{A}),
$$
where $Range$ refers to the range space and $d\phi$ denotes the Jacobian matrix.
This property is very important which is saying that
the tangent space $T_{\mathcal{M}}(A)$ provides a local vector space approximation of
the manifold $\mathcal{M}$. It is well known that this definition of tangent space is independent of
$\phi$. Moreover, we set
$$
\tilde{T}_{\mathcal{M}}(A)=T_{\mathcal{M}}(A)+A,
$$
i.e., $\tilde{T}_{\mathcal{M}}(A)$ is the affine linear manifold which is tangent to $\mathcal{M}$ at $A.$ Let the map $\phi$ be given as in \eqref{aff1}. Suppose that $A =\phi(x_{A})\in Im(\phi),$ the projection onto the tangent space of $\mathcal{M}$ at $A$ can be written as $P_{T_{\mathcal{M}}(A)}=M(M^*M)^{-1}M^*,$ where $M=d \phi(\phi^{-1}(A))$. Then the following result can be derived by the continuous of $d\phi$ and $\phi^{-1}$.

\begin{proposition}[Proposition 2.2 in \cite{andersson2013alternating}]\label{pro1}
Let $\mathcal{M}$ be a $\mathbb{C}^{1}$-manifold. Then $P_{T_{\mathcal{M}}(A)}$ is a continuous function of $A$.
\end{proposition}

The following proposition shows that the projections listed in \eqref{pj1} are locally well-defined.

\begin{proposition}[Proposition 2.3 in \cite{andersson2013alternating}]\label{pro2}
Let $\mathcal{M}$ be a $\mathbb{C}^{2}$-manifold, and let $A\in \mathcal{M}$ be given. Then there exists an $s>0$ such that for all $B\in \mathcal{B}_{\mathcal{K}}(A,s)$, there
exist a unique closest point in $\mathcal{M}.$ Denoting this point by $\pi(B),$ the map $\pi:~\mathcal{B}_{\mathcal{K}}(A,s)\rightarrow \mathcal{M}$ is $\mathbb{C}^{2}$. Moreover, $C\in \mathcal{M}\cap \mathcal{B}_{\mathcal{K}}(A,s)$ equals $\pi(B)$ if and only if $B-C\perp T_{\mathcal{M}}(C)$.
\end{proposition}

\subsection{The Alternating Projection Method}\label{ap}

It is well known that the convergence speed of the alternating projection method on linear manifolds is linear and decided by the angle between the two linear manifolds. Then, in order to generalize the alternating projection method to nonlinear manifolds, we first need to introduce the angle between two nonlinear manifolds. In this paper, the angle $\alpha(A)$ of $A\in \mathcal{M}=\mathcal{M}_{1}\cap \mathcal{M}_{2}$ is defined as
\begin{align}\label{df1}
\alpha(A)=cos^{-1}(\sigma(A)) ~ {\rm and}~
\sigma(A)=\lim_{\xi\rightarrow 0} \sup_{B_{1}\in F^{r}_{1}(A), B_{2}\in F^{r}_{2}(A)}
\left\{\frac{\left<B_{1}-A,B_{2}-A\right>}{\|B_{1}-A\|_{F}\|B_{2}-A\|_{F}}\right\},
\end{align}
with
$$
F_{1}^{\xi}(A)=
\{
B_1 \ | \ B_1 \in \mathcal{M}_{1}\backslash A,
\|B_1-A\|_{F}\leq \xi, B_{1}-A \bot T_{\mathcal{M}_{1}\cap \mathcal{M}_{2}}(A)
\},
$$
$$
F_{2}^{\xi}(A)=
\{
B_2 \ | \ B_2 \in \mathcal{M}_{2}\backslash A,
\|B_2-A\|_{F}\leq \xi, B_{2}-A \bot T_{\mathcal{M}_{1}\cap \mathcal{M}_{2}}(A)
\},
$$
and
$T_{\mathcal{M}_{1}\cap \mathcal{M}_{2}}(A)$ is
the tangent space of
$\mathcal{M}_{1} \cap \mathcal{M}_{2}$ at point $A$. Based on this definition, a point $A\in \mathcal{M}_{1}\cap \mathcal{M}_{2}$
is called a nontrivial intersection point when the angle is well defined. In addition, $A$ is tangential if $\alpha(A)=0$ and
nontangential if $\alpha(A)>0.$  Andersson and Carlsson \cite{andersson2013alternating} showed that if
$A$ is a nontangential intersection point of $\mathcal{M}_{1}$ and $\mathcal{M}_{2},$  there exist an $r>0$, such that for any point $B\in \mathcal{B}(A,r)$, the sequence of alternating projections
\begin{align*}
X_{0}=\pi_{1}(B),~X_{1}=\pi_{2}(X_{0}),~X_{2}=\pi_{1}(X_{1}), ~X_{3}=\pi_{2}(X_{2}),~....~,
\end{align*}
convergent to a point on $\mathcal{M}=\mathcal{M}_{1}\cap\mathcal{M}_{2}$, which is fairly close to the optimal point $\pi(B)$, see Figure 1(a).
In addition, the convergence rate is proved to be decided by the angle between the two manifolds.

\subsection{Projections Based on Tangent Spaces}

Alternating Projection (AP) method updates the sequence by projecting an initial point
back and forth between two manifolds. However, it can be computational expensive when a point is directly projected onto a manifold. For example, it is expensive to project a matrix onto the fixed rank manifold by the singular value decomposition (SVD) with a truncation out small singular values. Thus it is meaningful to find some new algorithm to reduce the computation complexity.
Note that matrix manifold algorithms based on the tangent space
have been widely studied in the literature (see for instance \cite{ccw2019,kwei} and their references therein).
These works motivated us to propose a Tangent spaces-based Alternating Projection (TAP) method
to reduce the computational cost.

\begin{figure}
\begin{center}
\includegraphics[height=2.5in,width=5.5in]{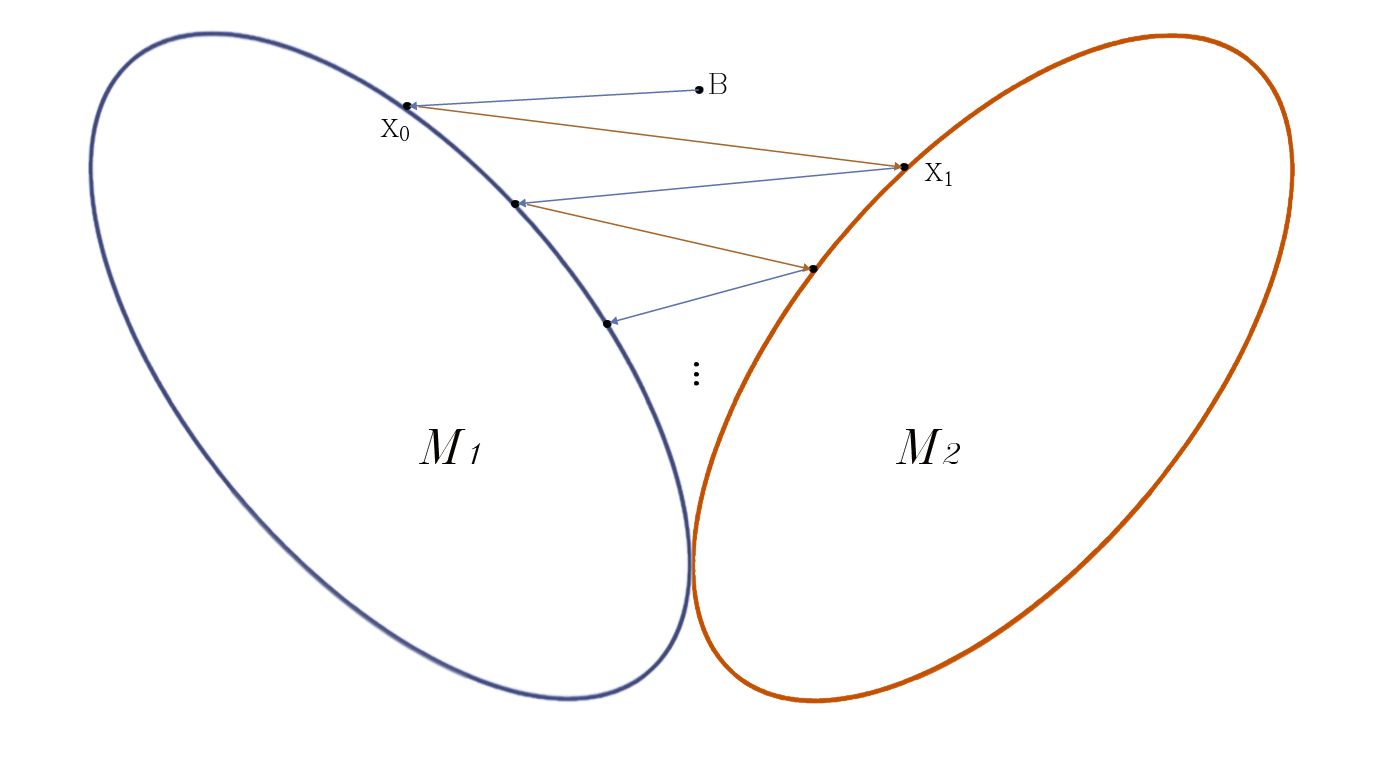}

\centerline{(a)}

\includegraphics[height=2.5in,width=5.5in]{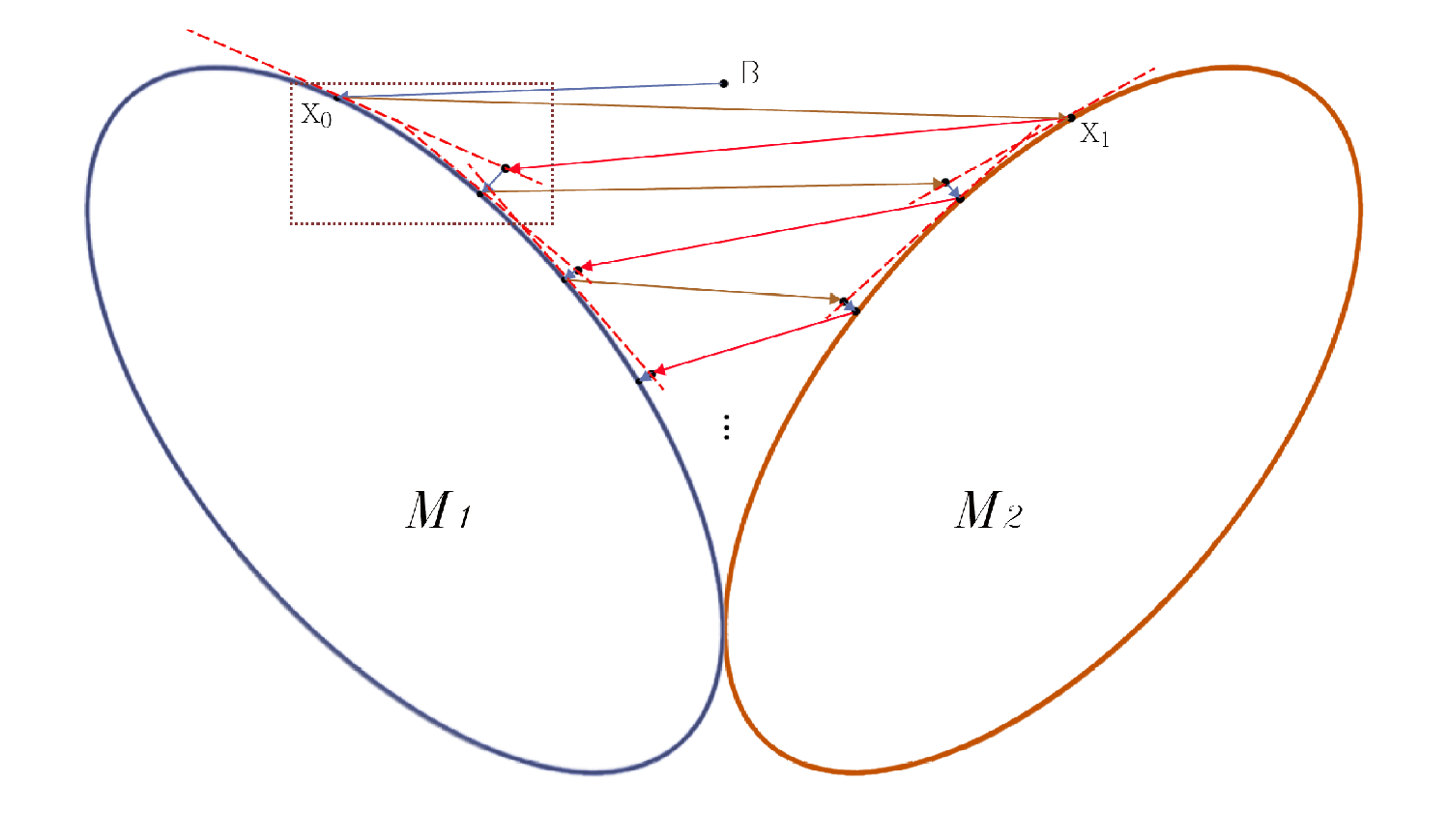}

\centerline{(b)}

\caption{The comparison between (a) the AP method and (b) the proposed TAP method.} \label{figure1}

\end{center}
\end{figure}

\begin{figure}
\begin{center}

\includegraphics[height=2.5in,width=5.5in]{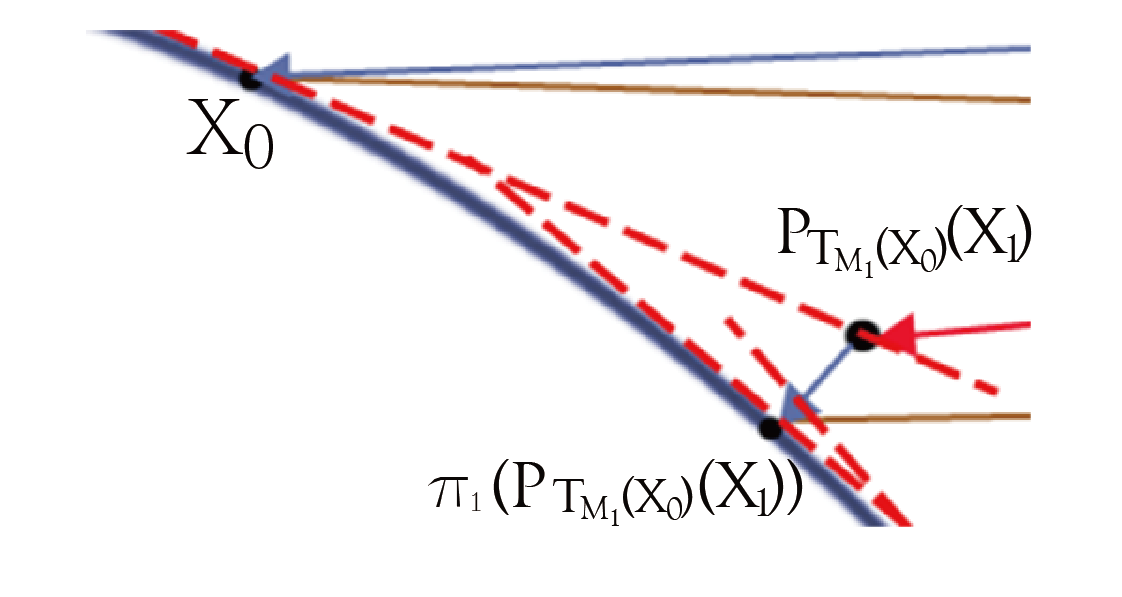}

\end{center}

\caption{The zoomed region in Figure 1(b).}
\label{figure1a}
\end{figure}

In Figure \ref{figure1} and Figure \ref{figure1a}, we demonstrate the proposed TAP method.
In the method,
the given point $B$ was first projected onto the manifold $\mathcal{M}_{1}$ to get a point $X_{0}$ by $\pi_{1}$, and then  $X_{1}$ is derived by projecting $X_{0}$ onto the manifold $\mathcal{M}_{2}$ by $\pi_{2}.$ The first two steps are same as the usual AP method given in  \cite{lewis2008alternating} and \cite{andersson2013alternating}. From the third step, the point $X_{1}$ is first projected onto the tangent space at $X_{0}$ of the manifold $\mathcal{M}_{1}$ by the  orthogonal projection $P_{T_{\mathcal{M}}(X_{0})}$,  and then the derived point is projected from the tangent space to the manifold to get the point $X_{2},$ after that the iterative sequence can be derived by similar projection method, which can be expressed as:
\begin{align}\label{eq1}
X_{0}=\pi_{1}(B),~X_{1}=\pi_{2}(X_{0}),~X_{2}=\pi_{1}(P_{T_{\mathcal{M}_{1}}(X_{0})} (X_{1})),~X_{3}=\pi_{2}(P_{T_{\mathcal{M}_{2}}(X_{1})} (X_{2})),....,
\end{align}
where $P_{T_{\mathcal{M}}(X_{i-1})}(X_{i}),~i=1,...,$ denotes the orthogonal projections of $X_{i}$ onto the tangent space of $\mathcal{M}$ at points $X_{i-1}, $ respectively. The proposed TAP method is given as the following algorithm.

\begin{algorithm}[h]
\caption{Tangent spaces-based Alternating Projection (TAP) Method} \label{ag1}
\textbf{Input: } Given a point ${A}\in \mathcal{B}(A_{0},s)$,
this algorithm computes the point on $\mathcal{M}_{1}\cap \mathcal{M}_{2}$ nearest $\pi(A)$. \\
~~1:  Initialize $X_{0}=A;$\\
~~2: $X_{1}=\pi_{1}(X_{0})$ and $Y_{1}=\pi_{2}(X_{1})$\\
~~3: for k=1,2,...,\\
~~4:  \quad ${ X}_{k+1}=\pi_{1}(P_{T_{\mathcal{M}_{1}}(X_{k})} (Y_{k}));$\\
~~5:  \quad ${ Y}_{k+1}=\pi_{2}(P_{T_{\mathcal{M}_{2}}(Y_{k})} (X_{k+1}));$\\
~~6: \textbf{end}\\
\textbf{Output:} ${ X}_k$ when the stopping criterion is satisfied.
\end{algorithm}

\subsection{Nonnegative Low Rank Matrix Approximation}

In this subsection, we demonstrate the proposed method by considering nonnegative low rank matrix approximation.
The nonnegative low rank matrix approximation is recently studied by Song and Ng in \cite{sm2019}.
The aim is to find a nonnegative low rank matrix $X$ such that $X \approx A$ such that their difference is as small as possible.
Mathematically, it can be formulated as the following optimization problem
\begin{equation}\label{pmain}
\min_{\rank({X})=r,{ X}\geq 0} \ \| { A}- { X}\|_{\textrm{F}}^{2}.
\end{equation}
In \cite{sm2019}, Song and Ng developed nonnegative low rank matrix
approximation by using the alternating projections on the $m\times n$ fixed-rank matrices manifold
\begin{align}\label{v1}
\mathcal{M}_{r}:=\left\{ X\in \mathbb{R}^{m\times n},~ \rank(X)= r\right\},
\end{align}
and the $m\times n$ non-negativity matrices manifold
\begin{align}
\mathcal{M}_{n}:=\left\{X\in \mathbb{R}^{m\times n},~ X_{i,j}\geq 0,~ i=1,...,m,~j=1,...,n\right\}.\label{v2}
\end{align}
The projection onto the fixed rank matrix set $\mathcal{M}_{r}$ is derived by the Eckart-Young-Mirsky theorem \cite{golub2012matrix} which  can be expressed as
\begin{align}\label{p1}
\pi_{1}({ X})=\sum_{i=1}^{r}\sigma_{i}(X) u_{i}(X) {v}_{i}^{T}(X),
\end{align}
where $\sigma_{i}(X)$ are first $r$ singular values of $X$, and $ u_{i}(X), v_{i}(X)$ are first $r$ columns of the unitary matrices of
$U(X)$ and $V(X)$.
The projection onto the nonnegative matrix set $\mathcal{M}_{n}$ is expressed as
\begin{align}\label{p2}
\pi_{2}({ X})=\left\{\begin{array}{c}
                     X_{ij}, ~~~ {\rm if} ~~ X_{ij}\geq 0, \\
                     0,     ~~~~~ {\rm if} ~~   X_{ij} < 0.
                   \end{array}
\right.
\end{align}
Then the sequence derived by the alternating projection method is convergent to a point on the intersection of the two manifolds
\begin{equation}\label{v3}
\mathcal{M}_{r} \cap \mathcal{M}_{n}=\left\{X\in \mathbb{R}^{m\times n}, ~ \rank(X)=r,~X_{ij}\geq 0, ~i=1,...,m,~j=1,...,n\right\},
\end{equation}
which is sufficiently close to the best nonnegative approximation, see \cite{sm2019}.
The main computational cost the above alternating projection method is to obtain the singular value decomposition $\pi_{1}({ X})$
at each iteration.

Let us consider the proposed TAP for solving nonnegative low rank matrix approximation problem.
Suppose that $k \ge 1$, $X_{k}$ and $Y_{k}$ are two consecutive terms in the sequence which are located on the manifold $\mathcal{M}_{r}$ and $\mathcal{M}_{n}$ respectively in Algorithm 1. Let
$X_{k}=U_k \Sigma_k V_k^{T}$ be the skinny SVD decomposition of $X_{k}$. It follows the results in
\cite{absil2009optimization} that the tangent space of $\mathcal{M}_{r}$ at $X_{k}$ can be expressed as
\begin{equation}
T_{\mathcal{M}_{r}(X_{k})}=\{ U_k W^{T} + Z V_k^{T} ~ | ~ W, Z \in \mathbb{R}^{n\times r} \text{are arbitrary}\}.
\end{equation}
For the given iterate $Y_{k}$, it can be easily
derived that the projections of $Y_{k}$ onto the subspace $T_{\mathcal{M}_{r}(X_{k})}$ and its orthogonal complement can be written as
\begin{align*}
P_{T_{\mathcal{M}_{r}(X_{k})}}(Y_{k})= U_k U_k^{T} Y_{k} + Y_{k} V_k V_k^{T} -
U_k U_k^{T} Y_{k} V_k V_k^{T},
\end{align*}
and
\begin{align*}
(I-P_{T_{\mathcal{M}_{r}(X_{k})}})(Y_{k})=(I- U_k U_k^{T}) Y_{k}(I- V_k V_k^{T}).
\end{align*}
Then $X_{k+1}=\pi_{1}(P_{T_{\mathcal{M}_{r}(X_{k})}}(Y_{k}))$ can be derived by projecting $P_{T_{\mathcal{M}_{r}(X_{k})}}(Y_{k})$ from the tangent space $T_{\mathcal{M}_{r}(X_{k})}$ to the manifold $\mathcal{M}_{r}$,
where $\pi_{1}$ is defined as \eqref{p1}. Compared with the AP method given in \cite{sm2019},
although an intermediate process is added in the TAP method, the matrix can be projected onto the $\mathcal{M}_{r}$ from a low dimensional subspaces which can reduce computational cost.  Computing the best rank-$r$ approximation of a non-structured $n\times n$ matrix, costs $O(n^{2}r)+n^2$ flops with a large hidden constant in front of $n^{2}r$.
In the proposed TAP method, the estimate of $X_{k+1}$ can be computed in a very efficient way.
Suppose that the $QR$ decompositions of $(I-U_k U_k^{T}) Y_{k} V_k$ and $(I-V_k V_k^{T})Y_{k}U_k$ are given as follows:
\begin{align*}
(I-U_k U_k^{T}) Y_{k} V_k = Q_k R_k
~ \text{and} ~
(I-V_k V_k^{T})Y_{k}U_k = \hat{Q}_k \hat{R}_k,
\end{align*}
respectively.
Recall that $U_k^{T}Q_k= V_k^{T} \hat{Q}_k =0$ and then by a direct computation, we have
\begin{align*}
& P_{T_{\mathcal{M}_{r}(X_{k})}}(Y_{k}) \\
=&~
U_k U_k^{T}) Y_{k} (I-V_k V_k^{T}) + (I-U_kU_k^{T}) Y_{k} V_k V_k^T +
U_k U_k^{T}) Y_{k} V_k V_k^{T} \\
=&~U_k \hat{R}_k^{T} \hat{Q}_k^{T} + Q_k R_k V_k^T + U_k U_k^T Y_k V_k V_k^T \\
=&~\left(
     \begin{array}{cc}
       U_k & Q_k \\
     \end{array}
   \right)\left(
            \begin{array}{cc}
              U_k^{T}Y_{k} V_k & \hat{R}_k^{T} \\
              R_k & 0 \\
            \end{array}
          \right)\left(
                      \begin{array}{c}
                        V_k^T \\
                        \hat{Q}_k^T \\
                      \end{array}
                    \right)\\
=:&~\left(
     \begin{array}{cc}
       U_k & Q_k \\
     \end{array}
   \right) M_k \left(
                      \begin{array}{c}
                        V_k^T \\
                        \hat{Q}_k^T \\
                      \end{array}
                    \right).
\end{align*}
Let $M_k= \Psi_k \Gamma_k \Phi_k^T$
be the SVD of $M_k$ which can be computed using $O(r^3)$ flops since
$M_k$ is a $2r\times 2r$ matrix.
Note that $\left(
            \begin{array}{cc}
              U_k & Q_k \\
            \end{array}
          \right)$
  and $\left(
            \begin{array}{cc}
              V_k & \hat{Q}_k \\
            \end{array}
          \right)$
  are orthogonal, then the SVD of $P_{T_{\mathcal{M}_{r}(X_{k})}}(Y_{k})= \Omega_k \Theta_k \Upsilon_k^T$
  can be computed by
 \begin{align*}
 \Omega_k =\left(
            \begin{array}{cc}
              U_k & Q_k \\
            \end{array}
          \right) \Psi_k,
          ~~
          \Theta_k = \Gamma_k ~~ \text{and} ~~
 \Upsilon_k=\left(
            \begin{array}{cc}
              V_k & \hat{Q}_k \\
            \end{array}
          \right) \Phi_k.
 \end{align*}
 It follows that  the overall computational cost of $\pi_{1}(P_{T_{\mathcal{M}_{r}(X_{k})}}(Y_{k}))$ can be expressed as two matrix-matrix multiplications between
an $n\times n$ matrix and an $n\times r$ matrix,  the $QR$ decomposition of two $n\times r$ matrices, and SVD of a $2r\times 2r$ matrix, and a few matrix-matrix multiplications between a $r\times n$ matrix and an $n\times r$ matrix or between an $n\times r$ matrix and a $r\times r$ matrix, which leading to a total of $4n^2r+O(r^2n+r^3)$ flops. The computational cost of each iteration of TAP method
is less than that of AP method.

\section{The Convergence Analysis}

In this section, we would like to show the convergence of the proposed TAP method.
We begin this section with some results given in \cite{andersson2013alternating} which are necessary for the proof of the convergence of Algorithm  \ref{ag1}.
The following lemma  says that the affine tangent-spaces are close to the manifold locally.

\begin{lemma}[Proposition 2.4 in \cite{andersson2013alternating}]\label{jia1}
Let $\mathcal{M}$ be a $\mathbb{C}^{2}$-manifold and $A\in \mathcal{M}$ be given. For each $\epsilon >0,$ there exists $s>0$ such that for all
$C\in \mathcal{B}(A,s)\cap \mathcal{M},$ we have:\\
$(i)~~ dist(D,\widetilde{T}_{\mathcal{M}}(C))\leq \epsilon \|D-C\|_{F},~\forall D\in \mathcal{B}(A,s)\cap\mathcal{M}.$\\
$(ii)~ dist(D,\mathcal{M})\leq \epsilon \|D-C\|_{F},~\forall D\in \mathcal{B}(A,s)\cap \widetilde{T}_{\mathcal{M}}(C)$.
\end{lemma}

Let $\rho_{j}(B) := P_{T_{\mathcal{M}_{j}}(\pi(B))}(B),~ j=1,2$ denote the maps that project $B\in \mathcal{B}(A,s)$ onto the tangent spaces of the manifolds $\mathcal{M}_{1}$ and $\mathcal{M}_{2}$ at the intersection point
$\pi(B),$ respectively. Then the following results can be derived by  Propositions \ref{pro1} and \ref{pro2}.

\begin{lemma}[Lemma 4.2 in \cite{andersson2013alternating}]\label{lem2}
The functions $\rho_{1}$ and $\rho_{2}$ are $\mathbb{C}^{1}$-maps in $\mathcal{B}(A,s_{0}).$ Moreover, we can select a number $s^{\epsilon}_{1}<s^{\epsilon}_{0}$ such that the image of $\mathcal{B}(A,s_{1}^{\epsilon})$ under $\rho_{1}, \rho_{2}, \pi, \pi_{1},\pi_{2}$ as well as any composition
of two of those maps, is contained in $\mathcal{B}(A,s_{0}^{\epsilon})$.
\end{lemma}

It follows from the definitions of $\rho_{j}$ and $\pi_{j}, j=1,2,$  that $\rho_{j}$ resembles $\pi_{j}$, but is slightly different. Andersson and Carlsson \cite{andersson2013alternating} estimated the difference between $\rho_{j}$ and $\pi_{j},j=1,2$, respectively.

\begin{lemma}[Proposition 4.3 in \cite{andersson2013alternating}]\label{jia2}
Suppose that $\epsilon>0$ with $\frac{1+\epsilon}{\sqrt{1-\epsilon^2}}<2$.
Given any $B\in \mathcal{B}(A,s^{\epsilon})$, we have
\begin{align*}
\|\pi_{j}(B)-\rho_{j}(B)\|_{F}<4\sqrt{\epsilon}\|B-\pi(B)\|_{F}, \quad j=1,2.
\end{align*}
\end{lemma}

Different from the results given in Lemma \ref{jia2},
we need to estimate the distances of $\pi_{j}(B)$ and $P_{T_{\mathcal{M}_{j}}(C_{j})}(B),$
where $C_{j}$ is a point on the manifold $\mathcal{M}_{j}$  and not necessary the
intersection point of $\mathcal{M}_{1}$ and $ \mathcal{M}_{2}$ for $j=1,2$.  We only list the results where the proof is similar to
the proof of Lemma \ref{jia2} given in \cite{andersson2013alternating}.

\begin{lemma}\label{ne1}
For each $\epsilon>0$ with $\frac{1+\epsilon}{\sqrt{1-\epsilon^2}}<2$.
Given any $B\in \mathcal{B}(A,s^{\epsilon})$ and $C_{j}\in \mathcal{M}_{j}$,
we have
\begin{align*}
\|\pi_{j}(B)-P_{T_{\mathcal{M}_{j}}(C_{j})}(B)\|_{F}<4\sqrt{\epsilon}\|B-C_{j}\|_{F},
\quad j=1,2.
\end{align*}
\end{lemma}

The following results are  used as the main tool to prove the convergence of the Alternating Projection
method given in \cite{andersson2013alternating}.

\begin{lemma}[Theorem 4.1 in \cite{andersson2013alternating}]\label{an3}
For each $\epsilon>0$,  there exist an $s>0$  such that for all $B\in \mathcal{B}(A,s)$,
we have
\begin{align}\label{an4}
\|\pi(\pi_{j}(B))-\pi(B)\|_{F}< \epsilon \|B-\pi(B)\|_{F}, \quad j=1,2.
\end{align}
\end{lemma}

We remark that for a given $\epsilon$ that the number $s$ given in Lemma \ref{an3}
may be different between the manifolds $\mathcal{M}_{1}$ and $\mathcal{M}_{2}$. Here we choose
$s$ such that \eqref{aff1} holds in all cases.
Moreover, the roles of $\pi_{1}$ and $\pi_{2}$, $\mathcal{M}_{1}$ and $\mathcal{M}_{2}$ in the proposed TAP method are equivalent, so
we choose $\pi_{1}$ and $\mathcal{M}_{1}$ as a special case. Then we can get the following results.

\begin{lemma}\label{new1}
For each $\epsilon>0$ given in Lemma \ref{ne1}, there exist $\alpha(\epsilon)>0,$
$\beta(\epsilon)>0$ and $s>0$
such that for all $B\in \mathcal{B}(A,s)$,
\begin{align}
\|\pi(\pi_{1}(P_{T_{\mathcal{M}_{1}}(C)}(B)))-\pi(B)\|_{F}<
\alpha(\epsilon) \|B-\pi(B)\|_{F}+ \beta(\epsilon) \|C-\pi(B)\|_{F},
\end{align}
where $C\in \mathcal{M}_{1}\cap \mathcal{B}(A,s)$,
$\pi_{1},\pi$ and $P_{T_{\mathcal{M}_{1}}(C)}$ stands for the projection
onto $\mathcal{M}_{1}$, $\mathcal{M}$ and the tangent space
$T_{\mathcal{M}_{1}}(C),$  respectively.
\end{lemma}

\begin{proof}
From Figure \ref{figure1}, we know that  in the third step of the TAP method, the given point is projected onto the tangent space of  $C\in \mathcal{M}_{1},$ i.e., $T_{\mathcal{M}_{1}}(C),$ instead of  onto $\mathcal{M}_{1}$ directly.
It follows Lemma \ref{ne1} that  for a given $\epsilon>0$ with $\frac{1+\epsilon}{\sqrt{1-\epsilon^2}}<2,$ there exist an $s^{\epsilon}>0$ such that for any  $B\in \mathcal{B}(A,s^{\epsilon}),$ $\pi_{1}(P_{T_{\mathcal{M}_{1}}(C)}(B))$ resembles $\pi_{1}(B)$, i.e.,
\begin{align*}
\|\pi_{1}(B)-P_{T_{\mathcal{M}_{1}}(C)}(B)\|_{F}<4\sqrt{\epsilon}\|B-C\|_{F}.
\end{align*}
Recall Lemma \ref{an3} and note that $\pi_{1}(B)\in \mathcal{M}_{1}$, then for each $\varepsilon^{\epsilon}$ there exist an $s^{\epsilon}>0$  such that for all $B\in \mathcal{B}(A,s^{\epsilon})$, we have
\begin{align*}
\|\pi(\pi_{1}(B))-\pi(B)\|_{F}<\varepsilon^{\epsilon} \|B-\pi(B)\|_{F}.
\end{align*}
$\pi_{1}(P_{T_{\mathcal{M}_{1}}(C)}(B))$ and $\pi_{1}(B)$ are all on the manifold $\mathcal{M}_{1},$ and $\pi_{1}(P_{T_{\mathcal{M}_{1}}(C)}(B))$ is the closest point to  $P_{T_{\mathcal{M}_{1}}(C)}(B)$ on the manifold $\mathcal{M}_{1},$ then
 \begin{align}\label{new3}
 \|P_{T_{\mathcal{M}_{1}}(C)}(B)-\pi_{1}(P_{T_{\mathcal{M}_{1}}(C)}(B))\|_{F}\leq \|P_{T_{\mathcal{M}_{1}}(C)}(B)-\pi_{1}(B)\|_{F}.
 \end{align}
 It follows from Lemma \ref{lem2} that there exist an $s^{\varepsilon}$ such that $\pi$ is $\mathbb{C}^2$ in $\mathcal{B}(A, s^{\varepsilon}).$
 Choose $s=\min(s^{\varepsilon},s^{\epsilon}),$ then for all $B\in \mathcal{B}(A,s)$,
\begin{align*}
      &\|\pi(\pi_{1}(P_{T_{\mathcal{M}_{1}}(C)}(B)))-\pi(B)\|_{F}\\
     =~&\|\pi(\pi_{1}(P_{T_{\mathcal{M}_{1}}(C)}(B)))-\pi(\pi_{1}(B))+\pi(\pi_{1}(B))-\pi(B)\|_{F}\\
 \leq ~&\|\pi(\pi_{1}(P_{T_{\mathcal{M}_{1}}(C_)}(B)))-\pi(\pi_{1}(B))\|_{F}+\|\pi(\pi_{1}(B))-\pi(B)\|_{F}\\
 \leq ~& \alpha\|\pi_{1}(P_{T_{\mathcal{M}_{1}}(C)}(B))-\pi_{1}(B)\|_{F}+\varepsilon\|B-\pi(B)\|_{F}\\
 \leq ~& \alpha\|\pi_{1}(P_{T_{\mathcal{M}_{1}}(C)}(B))-P_{T_{\mathcal{M}_{1}}(C)}(B)+P_{T_{\mathcal{M}_{1}}(C)}(B)-\pi_{1}(B)\|_{F}+\varepsilon\|B-\pi(B)\|_{F}\\
  \leq ~& \alpha\|\pi_{1}(P_{T_{\mathcal{M}_{1}}(C)}(B))-P_{T_{\mathcal{M}_{1}}(C)}(B)\|_{F}
  +\alpha\|P_{T_{\mathcal{M}_{1}}(C)}(B)-\pi_{1}(B)\|_{F}+\varepsilon\|B-\pi(B)\|_{F}\\
  \leq~& 2\alpha\|P_{T_{\mathcal{M}_{1}}(C)}(B)-\pi_{1}(B)\|_{F}+\varepsilon\|B-\pi(B)\|_{F}\\
  \leq ~&8\alpha\sqrt{\epsilon}\|B-C\|_{F}+\varepsilon\|B-\pi(B)\|_{F}\\
  \leq ~&8\alpha\sqrt{\epsilon}\|C-\pi(B))\|_{F}+(\varepsilon+8\alpha\sqrt{\epsilon})\|B-\pi(B)\|_{F}\\
  \leq~&\varepsilon^{\epsilon}_{1}\|B-\pi(B)\|_{F}+\varepsilon^{\epsilon}_{2}\|C-\pi(B)\|_{F}.
\end{align*}
The first part of the second inequality follows by the continuity of  $\pi,$ and the second part follows by Lemma \ref{an3}.
The first part of the fifth inequality as well as the sixth inequality follows by \eqref{new3} and Lemma  \ref{ne1}, respectively. Then the last inequality can be derived  by choosing
$\varepsilon+8\alpha\sqrt{\epsilon}\leq \varepsilon^{\epsilon}_{1}$ and $8\alpha\sqrt{\epsilon}\leq \varepsilon^{\epsilon}_{2}$.
 In particular, if $B=\pi_{2}(C)$, we can get $\varepsilon^{\epsilon}_{1}=\varepsilon$ and $\varepsilon^{\epsilon}_{2}=8\alpha\sqrt{\epsilon}$.
\end{proof}

 Recall the function $\sigma(A)$ given in \eqref{df1},  the following results show that the distances between $\pi_{j}(B)$ ($j=1,2$) and $\pi(B)$ are reduced in proportion to the angle between $\mathcal{M}_{1}$ and $\mathcal{M}_{2}$.

\begin{lemma}[Theorem 4.5 in \cite{andersson2013alternating}]\label{an5}
For each $c>\sigma(A),$  there exist an $s>0$ such that for all $B\in\mathcal{M}_{2}\cap \mathcal{B}(A,s)$,
we have
\begin{align}
\|\pi_{1}(B)-\pi(B)\|_{F}<c\|B-\pi(B)\|_{F}.
\end{align}
Moreover, the same holds true with the roles of $\mathcal{M}_{1}$ and $\mathcal{M}_{2}$ reversed.
\end{lemma}

In our case, the given point $B$ is firstly projected onto the tangent space of $\mathcal{M}_{1}$ at $C$ by the projector $P_{T_{\mathcal{M}_{1}}(C)},$ and then the derived point $P_{T_{\mathcal{M}_{1}}(C)}(B)$ is projected from the tangent space to the manifold to get $\pi_{1}(P_{T_{\mathcal{M}_{1}}(C)}(B)).$  Then the distance between  $\pi_{1}(P_{T_{\mathcal{M}_{1}}(C)}(B))$ and $\pi(B)$,
 can be estimated as follows.

\begin{lemma}\label{new2}
Suppose that $C\in \mathcal{M}_{1}$ and $B=\pi_{2}(C)\in \mathcal{M}_{2},$ for each $c>\sigma(A),$  there exist an $s>0$ such that for all $P_{T_{\mathcal{M}_{1}}(C)}(B)\in T_{\mathcal{M}_{1}}(C)\cap \mathcal{B}(A,s)$,
we have
\begin{align}
\|\pi_{1}(P_{T_{\mathcal{M}_{1}}(C)}(B))-\pi(B)\|_{F}<c\|B-\pi(B)\|_{F}.
\end{align}
Moreover, the same holds true with the roles of $\mathcal{M}_{1}$ and $\mathcal{M}_{2}$ reversed.
\end{lemma}

\begin{proof}
By  Proposition \ref{pro1} and Lemma \ref{lem2}, there exist an  $s_{0}$ such that $P_{T_{\mathcal{M}_{1}}}$, $P_{T_{\mathcal{M}_{1}\cap\mathcal{M}_{2}}}$ and $\pi$ are continuous functions
on $\mathcal{B}(A,s_{0}).$ Hence, we can pick $\alpha>0$ such that
\begin{align*}
\|P_{T_{\mathcal{M}_{1}}}(B)-P_{T_{\mathcal{M}_{1}}}(B^{'})\|_{F}\leq \alpha\|B-B^{'}\|_{F},\|\pi(B)-\pi(B^{'})\|_{F}\leq \alpha \|B-B^{'}\|_{F}
\end{align*}
and
\begin{align*}
\|P_{T_{\mathcal{M}_{1}\cap\mathcal{M}_{2}}}(B)-P_{T_{\mathcal{M}_{1}\cap\mathcal{M}_{2}}}(B^{'})\|_{F}\leq \alpha\|B-B^{'}\|_{F},
\end{align*}
for all $B,B^{'}\in \mathcal{B}(A,s_{0})$.
Fix $c_{1}$ such that $\sigma(A)<c_{1}<c,$ and pick an $s_{1}<s_{0}$ such that
\begin{align*}
sup\{\sigma(Z):Z\in \mathcal{M}_{1}\cap\mathcal{M}_{2}\cap \mathcal{B}(A,s_{1})\}<c_{1}
\end{align*}
Let $c_{2}>1$ such that $c_{2}c_{1}<c.$  Fix $\epsilon$ such that
\begin{align}\label{new6}
\frac{1+\epsilon}{\sqrt{1-\epsilon^2}}<2, ~~8\sqrt{\epsilon}(2+\alpha)\frac{c_{2}}{c_{2}-1}\|C\|_{F}<c\|B\|_{F}~~ \text{and}~~(1+4\sqrt{\epsilon})c_{2}c_{1}<c.
\end{align}
Then fix $s<s_{1}$ such that
\begin{align*}
\pi(\mathcal{B}(A,s))\subset \mathcal{B}(A,s_{1}),
\end{align*}
and let $B\in \mathcal{B}(A,s)\cap \mathcal{M}_{2}.$ There is no restriction to assume that $\pi(B)=0,$ which we do from now on.  Note that
$\pi(B)\in \mathcal{M}_{1}\cap \mathcal{M}_{2}\cap \mathcal{B}(A,s_{1}),$ so $\sigma(0)=\sigma(\pi(B))<c_{1}$.
In order to show
\begin{align*}
\frac{\|\pi_{1}(P_{T_{\mathcal{M}_{1}}(C)}(B))-\pi(B)\|_{F}}{\|B-\pi(B)\|_{F}}=\frac{\|\pi_{1}(P_{T_{\mathcal{M}_{1}}(C)}(B))\|_{F}}{\|B\|_{F}}<c,
\end{align*}
we need  the following auxiliary information.
Setting $B^{'}=P_{T_{\mathcal{M}_{2}}(\pi(B))}(B),$ $D=\pi_{1}(B),$ $D^{'}=P_{T_{\mathcal{M}_{2}}(\pi(B))}(B^{'})$ and $E=\pi_{1}(P_{T_{\mathcal{M}_{1}}(C)}(B)),$ then it is sufficient to show
\begin{align*}
\frac{\|\pi_{1}(P_{T_{\mathcal{M}_{1}}(C)}(B))\|_{F}}{\|B\|_{F}}=\frac{\|E\|_{F}}{\|D^{'} \|_{F}}\frac{\|B^{'} \|_{F}}{\|B\|_{F}}\frac{\|D^{'} \|_{F}}{\|B^{'} \|_{F}}<c.
\end{align*}
The values of the three fractions $\frac{\|E\|_{F}}{\|D^{'} \|_{F}}, \frac{\|B^{'} \|_{F}}{\|B\|_{F}}$  and $\frac{\|D^{'} \|_{F}}{\|B^{'} \|_{F}}$ will be derived independently in the sequel. Recall the definition of $B^{'}$ and   by Lemma \ref{jia2}, we have
\begin{align*}
\|B-B^{'}\|_{F}=\|\pi_{2}(B)-B^{'}\|_{F}<4\sqrt{\epsilon}\|B-\pi(B)\|_{F}.
\end{align*}
Similar to proof of Theorem 4.5 in \cite{andersson2013alternating}, we can get
\begin{align}\label{new5}
\frac{\|B^{'}\|_{F}}{\|B\|_{F}}\leq 1+4\sqrt{\epsilon}, ~~\frac{\|D^{'}\|_{F}}{\|B\|_{F}}\leq c_{1}
\end{align}
and
\begin{align*}
\|D-D^{'}\|_{F}
&=\|\rho_{1}(B^{'})-\pi_{1}(B)\|_{F}\leq \|\rho_{1}(B^{'})-\rho_{1}(B)\|_{F}+\|\rho_{1}(B)-\pi_{1}(B)\|_{F}\\
&\leq \alpha\|B^{'}-B\|_{F}+4\sqrt{\epsilon}\|B\|_{F}=4\sqrt{\epsilon}(1+\alpha)\|B\|_{F},
\end{align*}
respectively.
For $\frac{\|E\|_{F}}{\|D^{'} \|_{F}},$
\begin{align}
\|E-D^{'}\|_{F}
&=\|E-D+D-D^{'}\|_{F}\leq \|E-D\|_{F}+\|D-D^{'}\|_{F} \nonumber\\
&= \|E-P_{T_{\mathcal{M}_{1}}(C)}(B)+P_{T_{\mathcal{M}_{1}}(C)}(B)-D\|_{F}+\|D-D^{'}\|_{F}\nonumber\\
&\leq \|E-P_{T_{\mathcal{M}_{1}}(C)}(B)\|_{F}+\|P_{T_{\mathcal{M}_{1}}(C)}(B)-D\|_{F}+\|D-D^{'}\|_{F}\nonumber\\
&=\|\pi_{1}(P_{T_{\mathcal{M}_{1}}(C)}(B))-P_{T_{\mathcal{M}_{1}}(C)}(B)\|_{F}+\|P_{T_{\mathcal{M}_{1}}(C)}(B)-\pi_{1}(B)\|_{F}\nonumber\\
&~~~+\|\pi_{1}(B)-D^{'}\|_{F}. \label{song1}
\end{align}
By Lemma \ref{ne1}, and note that $B=\pi_{2}(C),\pi(B)$ are all on  $\mathcal{M}_{2},$ then
\begin{align}\label{song2}
\|P_{T_{\mathcal{M}_{1}}(C)}(B)-\pi_{1}(B)\|_{F}<4\sqrt{\epsilon}\|B-C\|_{F}\leq 4\sqrt{\epsilon}\|C-\pi(B)\|_{F}.
\end{align}
Similarly, $\pi_{1}(P_{T_{\mathcal{M}_{1}}(C_{1})}(B))$ and $\pi_{1}(B)$ are all on  $\mathcal{M}_{1}$, thus
\begin{align}\label{song3}
\|\pi_{1}(P_{T_{\mathcal{M}_{1}}(C)}(B))-P_{T_{\mathcal{M}_{1}}(C)}(B)\|_{F}\leq\|P_{T_{\mathcal{M}_{1}}(C)}(B)-\pi_{1}(B)\|_{F}\leq4\sqrt{\epsilon}\|C-\pi(B)\|_{F}.
\end{align}
Combining \eqref{song1}, \eqref{song2} and \eqref{song3}, we have
\begin{align}
\|E-D^{'}\|_{F}&<4\sqrt{\epsilon}\|C-\pi(B)\|_{F}+4\sqrt{\epsilon}\|C-\pi(B)\|_{F}+4\sqrt{\epsilon}(1+\alpha)\|B-\pi(B)\|_{F} \nonumber\\
&= 8\sqrt{\epsilon}\|C-\pi(B)\|_{F}+4\sqrt{\epsilon}(1+\alpha)\|B-C+C-\pi(B)\|_{F} \nonumber\\
&\leq 8\sqrt{\epsilon}\|C-\pi(B)\|_{F}+4\sqrt{\epsilon}(1+\alpha)(\|B-C\|_{F}+\|C-\pi(B)\|_{F})v \nonumber\\
&<8\sqrt{\epsilon}\|C-\pi(B)\|_{F}+8\sqrt{\epsilon}(1+\alpha)\|C-\pi(B)\|_{F}\nonumber\\
&=(2+\alpha)8\sqrt{\epsilon}\|C-\pi(B)\|_{F}.\label{song4}
\end{align}
For the value of $\frac{\|E\|_{F}}{\|D^{'}\|_{F}},$ there exist two case: one case is $\frac{\|E\|_{F}}{\|D^{'}\|_{F}}\leq c_{2}$ and the other one is
$\frac{\|E\|_{F}}{\|D^{'}\|_{F}}> c_{2}.$ If the later case is satisfied, by \eqref{song4}, we have
\begin{align}\label{e1}
&\|E\|_{F}-\|D^{'}\|_{F}<\|E-D^{'}\|_{F}<(2+\alpha)8\sqrt{\epsilon}\|C\|_{F}\nonumber\\
\Rightarrow& (c_{2}-1)\|D^{'}\|_{F}<(2+\alpha)8\sqrt{\epsilon}\|C\|_{F}\Rightarrow \|E\|_{F}<\frac{c_{2}}{c_{2}-1}(2+\alpha)8\sqrt{\epsilon}\|C_{1}\|_{F}.
\end{align}
Then a suitable $\epsilon$  can be chosen such that
\begin{align*}
\frac{c_{2}}{c_{2}-1}(2+\alpha)8\sqrt{\epsilon}\|C\|_{F}<c\|B\|_{F},
\end{align*}
in which case we are done.
For the other case, combining the results in \eqref{new5} and $\frac{\|E\|_{F}}{\|D^{'}\|_{F}}\leq c_{2}$,  we can get
\begin{align*}
\frac{\|E\|_{F}}{\|B\|_{F}}=\frac{\|E\|_{F}}{\|D^{'}\|_{F}}\frac{\|B^{'}\|_{F}}{\|B\|_{F}}\frac{\|D^{'}\|_{F}}{\|B^{'}\|_{F}}<(1+4\sqrt{\epsilon})c_{2}c_{1}<c,
\end{align*}
where the second inequality follows by \eqref{new6}.
\end{proof}

We can now list the main results as the following theorem.
\begin{theorem}\label{thm_convergence}
Suppose that $\mathcal{M}_{1}$, $\mathcal{M}_{2}$ and $\mathcal{M}=\mathcal{M}_{1}\cap \mathcal{M}_{2}$ are $\mathbb{C}^2$-manifold. Let $A_{0}$ be a nontangential point of $\mathcal{M}_{1}\cap \mathcal{M}_{2}$.
Then for any  given $\epsilon$ and
$1>c>\sigma({ A}_{0})$, there exist an
$\xi>0 $  such that for any ${ A}\in {\cal B}({ A}_{0},\xi)$ (the ball neighborhood of $A_0$ with radius $\xi$)
 the sequence $\{{X}_k\}_{k=0}^\infty$ generated by the alternating projection algorithm initializing from given $A$:
\begin{align*}
&X_{0}=\pi_{1}(A),X_{1}=\pi_{2}(X_{0}),X_{2}=\pi_{1}(P_{T_{\mathcal{M}_{1}}(X_{0})}(X_{1})),X_{3}=\pi_{2}(P_{T_{\mathcal{M}_{2}}(X_{1})}(X_{2})),...,\\
&X_{2k}=\pi_{1}(P_{T_{\mathcal{M}_{1}}(X_{2k-2})}(X_{2k-1})),X_{2k+1}=\pi_{2}(P_{T_{\mathcal{M}_{2}}(X_{2k-1})}(X_{2k})),...
\end{align*}
satisfies the following results:
\begin{enumerate}[(i)]
\item  converges to a point $X_\infty \in {\cal M}_{1}\cap {\cal M}_{2}$,
\item $\| { X}_\infty - \pi({ A}) \|_{F} \leq \epsilon \| {A} - \pi({ A}) \|_{F}$,
\item $\| { X}_\infty - {X}_k \|_{F} \leq {\rm const} \cdot c^k \| {A} - \pi({ A}) \|_{F}$.
\end{enumerate}
\end{theorem}
\begin{proof}

Assume that $\epsilon<1$ and $\sigma({ A}_{0})<c<1$.  Recall Lemma \ref{new1}, set
\begin{align*}
\varepsilon_{1}=\varepsilon=\frac{1-c}{2(3-c)}\epsilon,~~\varepsilon_{2}=\frac{1-c}{2+2\alpha}\epsilon,
\end{align*}
where $\alpha$ is a constant given as in the proof of Lemma \ref{new2}. Moreover, there exist some possibly distinct radii that  guarantee Lemma \ref{an5}-\ref{new2} are satisfied. Let $s$
 denote the minimum of these possibly  radii and pick
 $r<\frac{s(1-\epsilon)}{4(2+\epsilon)},$ so that $\pi(\mathcal{B}(A_{0},r))\subseteq \mathcal{B}(A_{0},s/4).$
Then $\|\pi(A)-A_{0}\|_{F}<s/4$ follows from the latter condition.
Denote $l=\|A-\pi(A)\|_{F},$ and note that
$$
l=\|A-A_{0}+A_{0}-\pi(A)\|_{F}\leq \|A-A_{0}\|_{F}+\|A_{0}-\pi(A)\|_{F}\leq r+s/4.
$$
As $\pi(A)\in \mathcal{M}_{1}\cap \mathcal{M}_{2}$ and note that $X_{0}=\pi_{1}(A),$ we have
\begin{align*}
\|X_{0}-A\|_{F}=\|\pi_{1}(A)-A\|_{F}\leq \|\pi(A)-A\|_{F}=l
\end{align*}
and
\begin{align*}
\|X_{0}-\pi(X_{0})\|_{F}\leq \|X_{0}-\pi(A)\|_{F}\leq \|X_{0}-A\|_{F}+\|A-\pi(A)\|_{F}\leq 2l.
\end{align*}
If \begin{align}\label{m1}
\{X_{k}\}_{k=0}^{k-1}\subseteq \mathcal{B}(A_{0},s)
\end{align}
 is satisfied, then by  Lemma \ref{new2}, we can get
\begin{align}\label{jia3}
\|X_{k}-\pi(X_{k})\|_{F}\leq\|X_{k}-\pi(X_{k-1})\|_{F}\leq c \|X_{k-1}-\pi(X_{k-1})\|_{F}.
\end{align}
Next we will show \eqref{m1} is satisfied by induction. Firstly, for $k=0,$
\begin{align*}
\|X_{0}-A_{0}\|_{F}\leq \|X_{0}-A\|_{F}+\|A-A_{0}\|_{F}\leq l+r/2\leq 2r+s/4\leq \frac{s(1-\epsilon)}{2(2+\epsilon)}+s/4<s.
\end{align*}
Assume that \eqref{m1} is satisfied when $n=k,$ then it follows from \eqref{jia3} that
\begin{align}\label{jia5}
\|X_{k}-\pi(X_{k})\|_{F}\leq c^{k}\|X_{0}-\pi(X_{0})\|_{F}\leq 2lc^{k}.
\end{align}
Note that
\begin{align}
 \|X_{k-2}-\pi(X_{k-1})\|_{F}
 =&\|X_{k-2}-\pi(\pi_{2}(X_{k-2}))\|_{F}\nonumber \\
 =&\|X_{k-2}-\pi(X_{k-2})+\pi(X_{k-2})-\pi(\pi_{2}(X_{k-2}))\|_{F}\nonumber\\
 \leq &\|X_{k-2}-\pi(X_{k-2})\|_{F}+\|\pi(X_{k-2})-\pi(\pi_{2}(X_{k-2}))\|_{F}\nonumber\\
 \leq & \|X_{k-2}-\pi(X_{k-2})\|_{F}+\alpha\|X_{k-2}-\pi_{2}(X_{k-2})\|_{F}\nonumber\\
 \leq &(1+\alpha)\|X_{k-2}-\pi(X_{k-2})\|_{F}.\nonumber
\end{align}
The second part of the second inequality follows by the continuous of the projection $\pi,$ the third inequality follows by $\|X_{k-2}-\pi_{2}(X_{k-2})\|_{F}\leq\|X_{k-2}-\pi(X_{k-2})\|_{F}.$
Applying Lemma \ref{new1} gives
\begin{align}
\|\pi(X_{k})-\pi(X_{k-1})\|_{F}
&<\varepsilon_{1} \|X_{k-1}-\pi(X_{k-1})\|_{F}+\varepsilon_{2} \|X_{k-2}-\pi(X_{k-1})\|_{F}\nonumber\\
&<\varepsilon_{1} \|X_{k-1}-\pi(X_{k-1})\|_{F}+\varepsilon_{2}(1+\alpha) \|X_{k-2}-\pi(X_{k-2})\|_{F}\nonumber\\
&\leq 2\varepsilon_{1}c^{k-1}l+2\varepsilon_{2}(1+\alpha)c^{k-2}l=(\varepsilon_{1}c+\varepsilon_{2}(1+\alpha))2c^{k-2}l.\label{new4}
\end{align}
Recall Lemma \ref{an4} and the inequality derived in \eqref{new4}, we have
\begin{align}\label{jia4}
\|\pi(X_{k})-\pi(A)\|_{F}&\leq \|\pi(A)-\pi(X_{0})\|_{F}+\|\pi(X_{1})-\pi(X_{0})\|_{F}+\sum_{j=2}^{k}\|\pi(X_{j})-\pi(X_{j-1})\|_{F} \nonumber\\
&\leq \varepsilon l+2\varepsilon l+\sum_{j=2}^{k}(\varepsilon_{1}c+\varepsilon_{2}(1+\alpha))2c^{j-2}l \nonumber\\
&\leq  3\varepsilon l+\frac{2(\varepsilon_{1}c+\varepsilon_{2}(1+\alpha))}{1-c}l=\frac{3\varepsilon(1-c)+2\varepsilon c+(1+\alpha)\varepsilon_{2}}{1-c}l\nonumber\\
&\leq \epsilon l.
\end{align}
Thus,
\begin{align*}
\|A_{0}-X_{k}\|_{F}&\leq \|A_{0}-\pi(A)\|_{F}+\|\pi(A)-\pi(X_{k})\|_{F}+\|\pi(X_{k})-X_{k}\|_{F}\\
&\leq s/4+  \epsilon l  +2l<s,
\end{align*}
which shows that \eqref{m1} is satisfied.

It follows from \eqref{new4} that the sequence $(\pi(B_{k}))_{k=1}^{\infty}$ is a Cauchy sequence, and then it converges to some point $B_{\infty}.$
Moreover, by \eqref{jia5} the sequence $(B_{k})_{k=1}^{\infty}$ must also converge, and the limit point is also $B_{\infty}.$ It follows that $B_{\infty}=\pi(B_{\infty}),$
then $(i)$ is concluded.
Moreover, by  taking the limit  \eqref{jia4} we can get $(ii).$ For $(iii).$ Note that
\begin{align}
\|\pi(B_{k})-B_{\infty}\|_{F}\leq \sum_{j=k+1}^{\infty}\|\pi(X_{j})-\pi(X_{j-1})\|_{F}\leq \frac{2l\varepsilon c^{k}}{1-c}+\frac{2(1+\alpha)l\varepsilon_{2} c^{k-1}}{1-c},
\end{align}
and combine with \eqref{jia5}, we can get
\begin{align*}
\|B_{k}-B_{\infty}\|_{F}
& \leq\|B_{k}-\pi(B_{k})\|_{F}+\|\pi(B_{k})-B_{\infty}\|_{F}
 \leq (2 l +\frac{2 l\varepsilon }{1-c}+\frac{2(1+\alpha)l\varepsilon_{2}}{1-c}) c^{k}\\
 &= \beta c^{k} l,
\end{align*}
with a constant $\beta$ as desired.
\end{proof}

In the next section, we will test the performance of the proposed TAP method.

\section{Numerical Examples}

In this section, numerical results are presented to show the effectiveness of the proposed TAP method (Algorithm 1).
There are two kinds of examples to be tested: nonnegative low rank matrix approximation and
low rank quaternion (color image) matrix approximation.
 All
the experiments are performed under Windows 7 and MATLAB R2018a
running on a desktop (Intel Core i7, @ 3.40GHz, 8.00G RAM).

\subsection{Nonnegative Low Rank Matrix Approximation}
\label{sec:main2}

In the first experiment, we randomly generated $n$-by-$n$ nonnegative matrices $A$
where their matrix entries follow a uniform distribution in between 0 and 1.
We employed the proposed TAP method and AP method \cite{sm2019}
to test the relative residual $\| A - X_{c} \|_F / \| A \|_F$,
where $X_c$ are the computed rank $r$ solutions by different methods.
For comparison, we also list the results by nonnegative matrix factorization
algorithms: A-MU \cite{gillis2012accelerated}, A-HALS \cite{gillis2012accelerated} and A-PG \cite{Lin06a}.

Tables \ref{table1a}  shows the relative
residuals of the computed solutions from the proposed TAP method and the other testing methods
for synthetic data sets of sizes 200-by-200, 400-by-400 and 800-by-800.
Note that there is no guarantee that
other testing NMF algorithms can determine the underlying nonnegative low rank factorization.
In the tables, it is clear that the testing NMF algorithms cannot obtain the underlying low rank factorization.
One of the reason may be that NMF algorithms can be sensitive to initial guesses.
In the tables, we illustrate this phenomena by displaying
the mean relative residual and the range containing both the minimum and the maximum relative residuals
by using ten initial guesses randomly generated.
We find in the table that the relative residuals computed by the TAP method is the same as those
by the AP method. It implies that the proposed TAP method can achieve the same accuracy of classical
alternating projection. According to the tables, the relative residuals by
both TAP and AP methods are always smaller than the minimum relative residuals by the testing NMF algorithms.
In addition, we report the computational time (seconds calculated by MATLAB) in the tables. We see that the
computational time required by the proposed TAP method is less than that required by AP method.

\begin{table}
 \begin{center}
  \setlength{\abovecaptionskip}{-1pt}
       \setlength{\belowcaptionskip}{-1pt}
\caption{The comparison of different algorithms.} \label{table1a}

\scalebox{0.8}{
\begin{tabular}{|l||c|c|c|}
\hline
       & \multicolumn{3}{c|}{200-by-200 matrix} \\
\cline{2-4}
Method             & $r=10$            & $r=20$               & $r=40$ \\
\hline
TAP                &  0.4576          &  0.4161               & 0.3247         \\
Time               &  0.42          &  0.48               & 0.38        \\
\hline
AP                 &  0.4576          &  0.4161               & 0.3247    \\
Time               &  0.66          &  0.66               & 0.42     \\
\hline
A-MU (mean)        &  0.4592              &  0.4249      &     0.3733            \\
A-MU (range)       & [0.4591, 0.4593] & [0.4246, 0.4251] & [0.3729, 0.3737]  \\
Time (mean)        &  8.32              & 9.54                  & 15.34                \\
Time (range)       & [8.00, 8.81]     &  [9.41, 9.61]         &  [14.72, 15.75]               \\
\hline
\hline
A-HALS (mean)       &      0.4591          &    0.4246                    &  0.3717 \\
A-HALS (range)      &     [0.4590, 0.4593] &     [0.4244, 0.4247]         &  [0.3714, 0.3719] \\	
Time (mean)         &     1.09            &    1.95                   &   4.01         \\
Time (range)        &    [0.98, 1.22]    &   [1.86, 2.05]         &  [3.86, 4.13]            \\
\hline
A-PG (mean)          &  0.4591               & 0.4244                     & 0.3717 \\
A-PG (range)         & [0.4590,0.4592]       &  [0.4243, 0.4246]         &  [0.3715, 0.3719]	 \\		
Time (mean)          &  14.77             & 16.24                      &  21.52          \\
Time (range)         &  [14.50,15.03]     &  [15.81, 16.55]        & [21.02, 21.77]           \\
\hline
\hline
       & \multicolumn{3}{c|}{400-by-400 matrix} \\
\cline{2-4}
Method                & $r=20$                & $r=40$               & $r=80$ \\
\hline
TAP                   & 0.4573                   & 0.4161                     & 0.3421           \\
Time                  & 1.55                 & 1.32                     &  1.10          \\
\hline
AP                     &0.4573                   & 0.4161                    &  0.3421  \\
Time                   & 2.95                  &2.47                      & 1.68           \\
\hline
A-MU (mean)            &  0.4606              & 0.4301                           &  0.3857 \\
A-MU (range)          &  [0.4605, 0.4607]     & [0.4300, 0.4302]                 & [0.3856, 0.3860]  \\
Time (mean)           &  37.80              & 38.72                          & 46.41           \\
Time (range)          &  [36.67, 39.03]   &  [38.21, 39.18]              & [45.87, 48.28]           \\
\hline
\hline
A-HALS (mean)          &0.4604                    &0.4295                   &0.3836   \\
A-HALS (range)         &[0.4603, 0.4605]          &[0.4294, 0.4296]         &[0.3833, 0.3838]    \\	
Time (mean)            & 3.10                   &7.40                   & 19.67           \\
Time (range)           & [3.03, 3.25]         &[7.12, 7.60]         & [19.04, 20.61]           \\
\hline
A-PG (mean)           &0.4604                &0.4297                         &0.3850   \\
A-PG (range)          &[0.4604, 0.4605]      &[0.4296, 0.4298]               &[0.3847, 0.3853] 	 \\		
Time (mean)           & 51.68              &60.80                        &61.95            \\
Time (range)          & [51.04, 52.26]   & [60.62, 61.01]            & [61.34, 62.64]           \\
\hline
\hline
       & \multicolumn{3}{c|}{800-by-800 matrix} \\
\cline{2-4}
Method       & $r=40$             & $r=80$                &   $r=160$ \\
\hline
TAP          &0.4550                    &0.4144                      &0.3412            \\
Time         &7.14                    &4.84                      &4.80            \\
\hline
AP           &0.4550                    &0.4144                      &0.3412  \\
Time         &15.84                   &9.55                      &7.11            \\
\hline
A-MU (mean)   &0.4608                 &0.4350                       &0.3984  \\
A-MU (range)  &[0.4607, 0.4609]       &[0.4349, 0.4351]             &[0.3982, 0.3986]   \\
Time (mean)   & 60.29               &60.96                      &61.31            \\
Time (range)  & [60.03, 60.65]     &[60.61, 61.58]           &[60.72, 61.94]            \\
\hline
\hline
A-HALS (mean)    &0.4605                       &0.4336                              &0.3984 \\
A-HALS (range)   &[0.4604,0.4605]              &[0.4335, 0.4336]                     &[0.3982, 0.3986] \\	
	Time (mean)  &18.54                      &47.70                                 &61.30            \\
Time (range)     &[17.76, 19.48]           &[43.33, 52.75]                      & [60.71, 61.91]           \\
\hline
A-PG (mean)      &0.4606                    &0.4343                                 & 0.4007 \\
A-PG (range)     &[0.4606, 0.4607]          &[0.4342, 0.4344]                       & [0.4005, 0.4012]	 \\		
Time (mean)      & 60.38                  & 61.26                               &  61.76          \\
Time (range)     &[60.12, 60.79]         &[60.78, 61.81]                     & [61.29 62.48]           \\
\hline
\end{tabular}}
\end{center}
\end{table}

Moreover, we considered the CBCL face database \cite{face}.
In the face database, there are $m=2469$ facial images, each consisting of $n=19\times19 = 361$ pixels, and constituting a face image matrix $A\in \mathbb{R}_{+}^{361 \times 2469}$. We tested several values of $r=20,40,60,80,100$ for nonnegative low rank minimization
and compared the proposed TAP method with the other algorithms.
In the testing NMF algorithms, we used 10 different initial guesses and report the results of
mean relative residuals in the table.
We see from Table \ref{table1e} that the relative residuals computed by the proposed
TAP method and the AP method is smaller than the mean relative residuals by the testing NMF algorithms.
Again the computational time required by the proposed TAP method is smaller than
that by the AP method.

\begin{table}
\centering
\caption{The relative residuals and computational time (in seconds) by different algorithms for face data matrix.}\label{table1e}
  \scalebox{0.9}{
\begin{tabular}{|l||c|c|c|c|c|}
  \hline
       & \multicolumn{5}{c|}{361-by-2649 matrix} \\
\cline{2-6}
Method          & $r=20$                  & $r=40$                &   $r=60$              & $r=80$              & $r=100$\\
\hline
TAP              &0.1170                  & 0.0839                &0.0645                &0.0529                &0.0438   \\
Time             &26.27                   & 13.97                 &9.98                  &7.87                  &7.26   \\
\hline
AP               & 0.1170                 & 0.0839                &0.0645                 &0.0529               &0.0438  \\
Time             & 70.50                  & 34.02                 &21.77                  &15.22                & 11.98 \\
\hline
A-MU (mean)      &0.1223                  & 0.0923                 &0.0756                &0.0645                &0.0561  \\
A-MU (range)     &[0.1219, 0.1229]        & [0.0918, 0.0934]       &[0.0752, 0.0760]      &[0.0640, 0.0649]      & [0.0557, 0.0564] \\
Time (mean)      & 60.29                  & 60.71                 &61.11                  & 61.48                & 61.87  \\
Time (range)     & [60.03, 60.65]         & [60.31,60.17]         &[60.81, 61.17]         &[60.93, 62.01]        & [61.42, 62.64]       \\
\hline
\hline
A-HALS (mean)    & 0.1220                 & 0.0919                &0.0720                 &0.0595                  & 0.0503\\
A-HALS (range)   & [0.1218, 0.1223]       & [0.0916, 0.0922]      &[0.0718,0.0722]        &[0.0594, 0.0598]        & [0.0502, 0.0505]  \\	
	Time (mean)  & 21.62                  & 39.40                 &54.88                  &61.22                   &105.97 \\
Time (range)     & [20.71, 22.78]         & [38.34,40.40]         &[52.15, 59.98]         &[60.73, 61.64]          & [103.03, 107.45]   \\
\hline
A-PG (mean)      &  0.1223                &0.0901                 & 0.0781                 &0.0687                  &0.0625  \\
A-PG (range)     & [0.1219, 0.1229]      &[0.0899,0.0904]         & [0.0776, 0.0787]       &[0.0682, 0.0692]        &[0.0618, 0.0632] \\		
Time (mean)      & 60.38                 & 60.65                  &60.08                   &61.22                   &61.04  \\
Time (range)     & [60.09, 60.71]        & [60.31, 61.17]         &[60.51, 61.39]          &[60.73, 61.64]         &[60.71,61.43]      \\
\hline
\end{tabular} }
\end{table}

\subsection{Low Rank Color Image Approximation}

Nowadays, color images appear commonly in many image processing applications.
The use of quaternion matrices for color images representation
has been studied in the literature, see \cite{le2004singular,,ell2006hypercomplex,han2013color,pei2001efficient,sangwine1996fourier,subakan2011quaternion}.
A color image contains red, blue and green channels, the quaternion approach is to encode the red, green and blue channel pixel values on
the three imaginary parts of a quaternion. The main advantage is that color images can be studied and processed holistically as a vector field, see \cite{ell2006hypercomplex,sangwine1996fourier,subakan2011quaternion}.
We can make use of quaternion matrix to represent color images and study the optimal rank-$r$ approximation of color image in terms of Frobenius norm. A low rank approximation of a purely quaternion matrix (red, green and blue channels color image)
can be obtained by using the quaternion singular value decomposition \cite{zhang1997quaternions}.
However, this approximation may not be optimal in the sense that the resulting approximation matrix may not be purely quaternion, i.e.,
it may contain the real component and it is not referred to a color image.
Here the following optimization problem is considered:
\begin{equation}\label{pmain}
\min_{\rank( {\bf X} )=r, \re( {\bf X} )=0}\| {\bf A} - {\bf X}\|_{F}^{2},
\end{equation}
where ${\bf A}$ is a given purely quaternion matrix and
$\re( {\bf X})$ stands for the real part of ${\bf X}$.
Here, we  employ  two color images ``peppafamily'' and ``pepper'' with sizes $256\times 256$ to compare TAP and AP in terms of residuals and time.
The original colorimages and their ranks $6,12,18$ and $24$ approximations by TAP and AP methods are listed in  Figure \ref{fig1}. We see from the figure that the low rank approximation of color images by the TAP
method and the AP method are about the same in terms of visual quality.
Their relative residuals and their computation time are shown in the Table \ref{newtable}.
According to the table, the relative residuals by both TAP and AP methods are nearly the same, however the computational time of the proposed TAP method is much less than
that required by AP method.

\begin{table}
\centering
\caption{The relative residuals and computational time (in second) by different algorithms for two color image matrices.}\label{newtable}
  \scalebox{0.9}{
\begin{tabular}{|l||c|c|c|c|c|c|c|c|}
  \hline
       & \multicolumn{4}{c|}{`peppafamily' color image matrix}  & \multicolumn{4}{c|}{`pepper' color image matrix}\\
\cline{2-9}
Method          & $r=6$                  & $r=12$                &   $r=18$              & $r=24$  &     $r=6$                  & $r=12$                &   $r=18$              &  $r=24$        \\
\hline
TAP              &0.1321                  & 0.1039                &0.0861                 &0.0734          &0.2267                  & 0.1648                &0.1317                &0.1112           \\
Time             &240.25                   &448.58                &710.58                 &930.01         & 270.44                 &458.05                &754.25                &919.98          \\
\hline
AP               & 0.1320                 & 0.1038                &0.0861                 &0.0733          &0.2267                  & 0.1646               &0.1317                &  0.1112      \\
Time             & 6856.07                & 6709.38               &7044.16                &6609.02         &6943.01                  & 6654.63                &6869.88               &6545.88        \\
\hline
\end{tabular} }
\end{table}

\begin{figure}[h]
\centering
\begin{minipage}{0.24\linewidth}\subfigure[{Original}]{\includegraphics[width=0.75\textwidth]{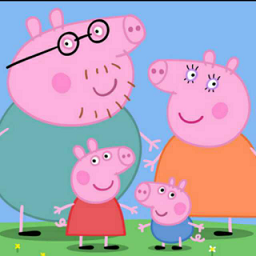}}
\end{minipage}
\begin{minipage}{0.75\linewidth}
\subfigure[{TAP, $r=6$}]{\includegraphics[width=0.24\textwidth]{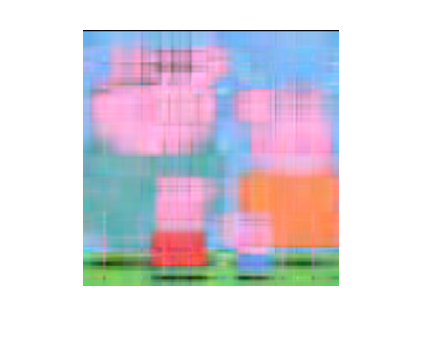}}
\subfigure[{TAP, $r=12$}]{\includegraphics[width=0.24\textwidth]{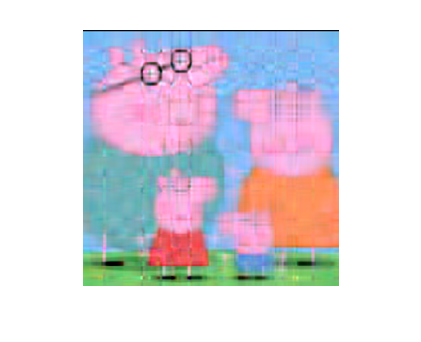}}
\subfigure[{TAP, $r=18$}]{\includegraphics[width=0.24\textwidth]{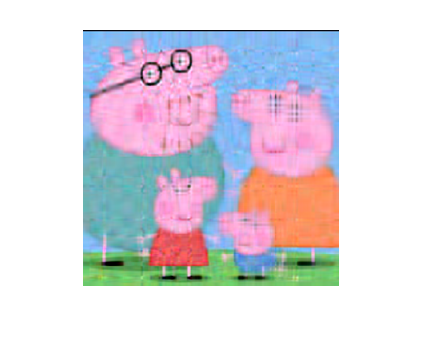}}
\subfigure[{TAP, $r=24$}]{\includegraphics[width=0.24\textwidth]{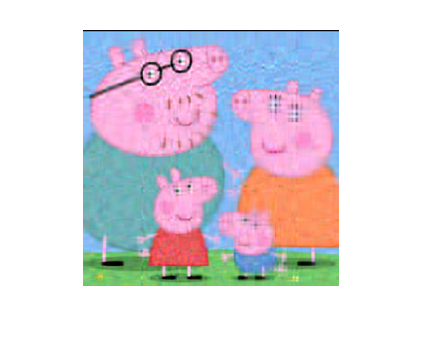}}\\

\subfigure[{AP, $r=6$}]{\includegraphics[width=0.24\textwidth]{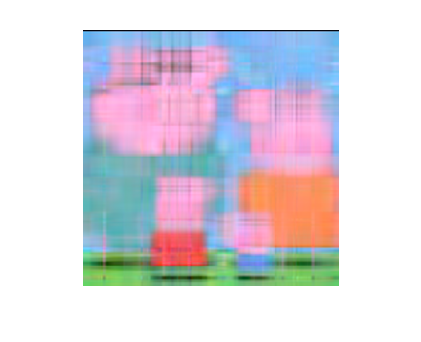}}
\subfigure[{AP, $r=12$}]{\includegraphics[width=0.24\textwidth]{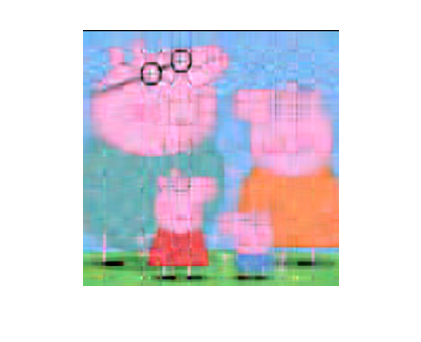}}
\subfigure[{AP, $r=18$}]{\includegraphics[width=0.24\textwidth]{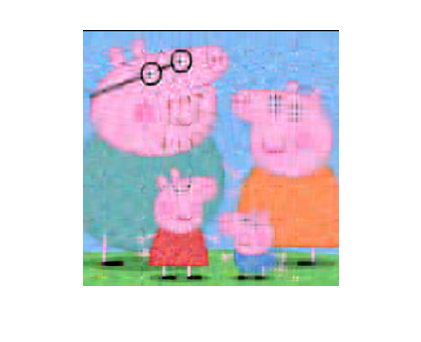}}
\subfigure[{AP, $r=24$}]{\includegraphics[width=0.24\textwidth]{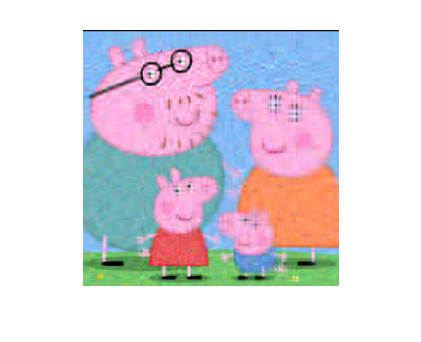}}\\

\end{minipage}

\setcounter{subfigure}{0}
\begin{minipage}{0.24\linewidth}\subfigure[{Original}]{\includegraphics[width=0.75\textwidth]{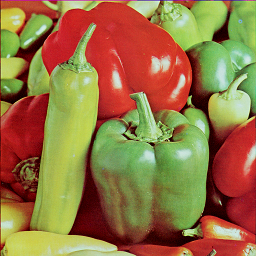}}
\end{minipage}
\begin{minipage}{0.75\linewidth}
\subfigure[{TAP, $r=6$}]{\includegraphics[width=0.24\textwidth]{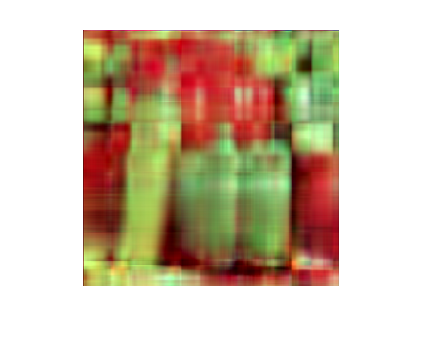}}
\subfigure[{TAP, $r=12$}]{\includegraphics[width=0.24\textwidth]{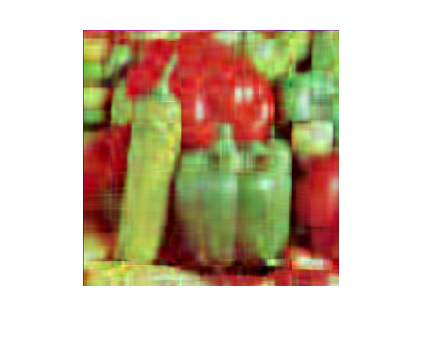}}
\subfigure[{TAP, $r=18$}]{\includegraphics[width=0.24\textwidth]{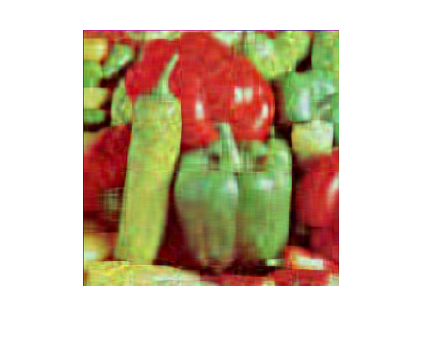}}
\subfigure[{TAP, $r=24$}]{\includegraphics[width=0.24\textwidth]{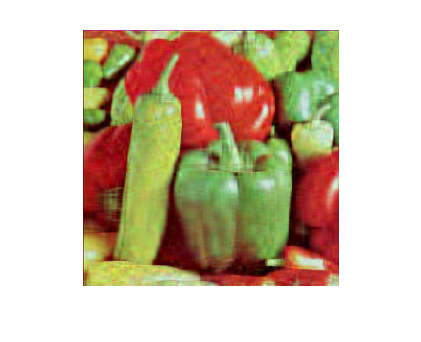}}\\

\subfigure[{AP, $r=6$}]{\includegraphics[width=0.24\textwidth]{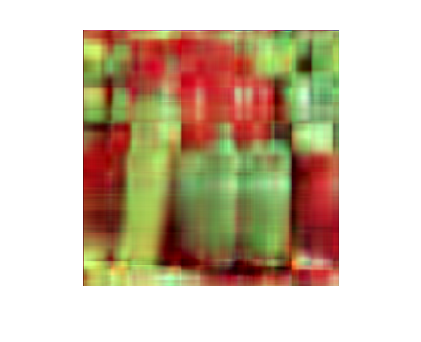}}
\subfigure[{AP, $r=12$}]{\includegraphics[width=0.24\textwidth]{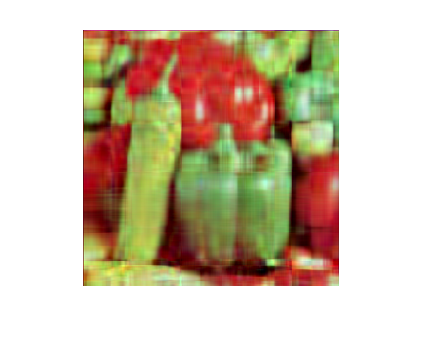}}
\subfigure[{AP, $r=18$}]{\includegraphics[width=0.24\textwidth]{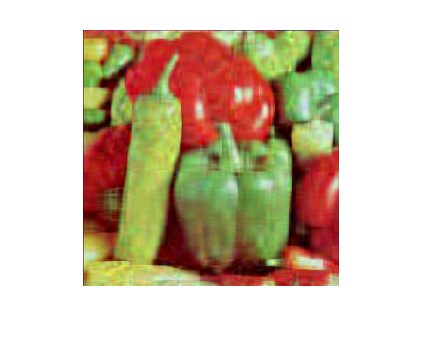}}
\subfigure[{AP, $r=24$}]{\includegraphics[width=0.24\textwidth]{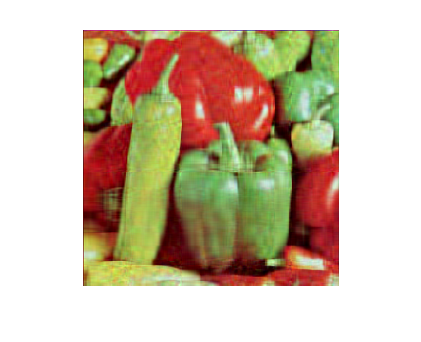}}\\

\end{minipage}

\caption{The rank $6,12,18,24$ approximations of `peppafamily' and `pepper' by TAP and AP methods, respectively
}
\label{fig1}
\end{figure}


\section{Conclusion}\label{sec:clu}

In this paper, we study alternating projections on nontangential manifolds
based on the tangent spaces. We have shown that
the sequence generated by alternating projections on two nontangential manifolds
based on tangent spaces, converges linearly to a point in the intersection of
the two manifolds where the convergent point is close to the optimal solution.
Numerical examples based nonnegative low rank matrix approximation and low
rank image quaternion matrix (color image) approximation are given to demonstrate
that the performance of the proposed method is better than that of the
classical alternating projection method in terms of computational time.

As a future research work, it is interesting to study applications involved
the projections of manifolds (for example face recognition). Moreover,
In many applications, researchers have suggested to use the other
norms (such as $l_1$ norm) in data fitting instead
of Frobenius norm to deal with other machine learning applications. It is
necessary to develop the related algorithms for such manifold optimization problems.


\begin{thebibliography}{99}

\bibitem{absil2009optimization}
P. Absil, R.~Mahony and R.~Sepulchre, Optimization algorithms on
  matrix manifolds, Princeton University Press, 2009.

\bibitem{andersson2013alternating}
F.~Andersson and M.~Carlsson,  {\em Alternating projections on nontangential manifolds, Constructive approximation,} 38 (2013) 489-525.


\bibitem{berger1988}
 M. Berger and B. Gostiaux, Differential Geometry: Manifolds, Curves and Surfaces. Springer, Berlin, 2012.

 \bibitem{le2004singular}
N.~Bihan and J.~Mars, {\em Singular value decomposition of quaternion
  matrices: a new tool for vector-sensor signal processing,} Signal processing,
  84 (2004) 1177-1199.

 \bibitem{ccw2019}
H. Cai, J. Cai and K. Wei, {\em Accelerated alternating projections for robust principal component analysis}, Journal of machine learning research 20 (2019) 1-33.
\bibitem{chenchu1996}

X. Chen and M. Chu, {\em On the least squares solution of inverse eigenvalue problems,} SIAM J. Number. Anal. 33 (1996) 2417-2430.

\bibitem{combettes1993signal}
 P. Combettes, {\em Signal recovery by best feasible approximation},
  IEEE transactions on Image Processing, 2 (1993) 269-271.

  \bibitem{ell2006hypercomplex}
 T. Ell and S. Sangwine,  {\em Hypercomplex fourier transforms of
  color images,} IEEE Transactions on image processing, 16 (2006) 22--35.


\bibitem{gillis2012accelerated}
N. Gillis and F. Glineur, {\em Accelerated multiplicative updates and hierarchical ALS algorithms for nonnegative matrix factorization,}
\emph{Neural computation} 24(4) (2012) 1085-1108.

\bibitem{golub2012matrix}
G. Golub and C. Van~Loan, Matrix Computations, vol.~3, JHU
  press, 2012.

\bibitem{grigoriadis1994application}
 K. Grigoriadis, A. Frazho and R. Skelton,  {\em Application of
  alternating convex projection methods for computation of positive toeplitz
  matrices,} IEEE transactions on signal processing, 42 (1994) 1873-1875.

\bibitem{grigoriadis1996}
 K. Grigoriadis and R. Skelon, {\em Low-order control design for LMI problems using alternating projection methods,} Automatica 32 (1996) 1117-1125.

\bibitem{Griogoriadis2000}
K. Grigoriadis and E. Beran, { Alternating projection algorithms for linear matrix inequalities problem with rank constraints, Adv. Linear Matrix Inequality Methos in Control,} SIAM, Philadelphia, 256-267.


\bibitem{hamaker1978angles}
 C.~Hamaker and D.~Solmon,  {\em The angles between the null spaces of x
  rays,} Journal of mathematical analysis and applications, 62 (1978)
  1-23.

\bibitem{han2013color}
 X.~Han, J.~Wu, L.~Yan, L.~Senhadji and H.~Shu, {\em Color image recovery
  via quaternion matrix completion}, in 2013 6th International Congress on
  Image and Signal Processing (CISP), vol.~1, IEEE (2013), 358--362.

\bibitem{higham2002computing}
N. Higham, {\em Computing the nearest correlation matrix -- a problem
  from finance}, IMA journal of Numerical Analysis, 22 (2002) 329--343.


\bibitem{kayalar1988error}
 S.~Kayalar and H. Weinert,  {\em Error bounds for the method of
  alternating projections, Mathematics of Control,} Signals and Systems, 1
  (1988) 43--59.

\bibitem{Lin06a}
 C. Lin, {\em Projected gradient methods for nonnegative matrix factorization,} {Neural computation} 19(10) (2007)  2756-2779.


\bibitem{lee1997conformal}
 S.~Lee, P. Cho, R. Marks and S.~Oh,  {\em Conformal radiotherapy
  computation by the method of alternating projections onto convex sets,}
  Physics in Medicine \& Biology, 42 (1997) 1065.

\bibitem{levi1983signal}
{A.~Levi and H.~Stark}, {\em Signal restoration from phase by projections
  onto convex sets,} JOSA, 73 (1983) 810-822.

\bibitem{lewis2008alternating}
{A. Lewis and J.~Malick},  {\em Alternating projections on manifolds,}
  Mathematics of Operations Research, 33 (2008) 216-234.


\bibitem{Neuman1950}
J. von Neumann, Functional Operators, vol. II: The Geometry of Orthogonal Spaces, Princeton University Press, Princeton, 1950.

\bibitem{orsi2006}
R. Orsi, U. Helmke and J. Moore, {\em A Newton-like method for solving rank constrained linear matrix inequalityies,} Automatica 42 (2006) 1875-1882.

 \bibitem{pei2001efficient}
 S. Pei, J. Ding and J. Chang, {\em Efficient implementation of
  quaternion fourier transform, convolution, and correlation by 2-d complex
  fft}, IEEE Transactions on Signal Processing, 49 (2001) 2783-2797.


\bibitem{sangwine1996fourier}
 S. Sangwine,  {\em Fourier transforms of colour images using quaternion
  or hypercomplex numbers,} Electronics letters, 32 (1996) 1979-1980.

 \bibitem{schwarz1986}
H. Schwarz, {\em \"{U}ber einige Abbildungsaufgaben,}
Gesammelte Mathematische Abhandlungen 11 (1869) 65-83.

\bibitem{sm2019}
G. Song and M. Ng, {\em Nonnegative Low Rank Matrix Approximation for Nonnegative Matrices,}
Applied Mathematics Letter.

 \bibitem{subakan2011quaternion}
{\"O}. Subakan and B. Vemuri, {\em A quaternion framework for color
  image smoothing and segmentation}, International Journal of Computer Vision,
  91 (2011) 233-250.

\bibitem {kwei}
K. Wei, J. Cai, T. Chan and S. Leung, {\em Guarantees of Riemanninan optimization for low rank matrix completion,} Inverse Problems and Imaging, 14(2) (2020) 233-265.

\bibitem{widrow1987adaptive}
B.~Widrow, {\em Adaptive inverse control}, in Adaptive Systems in Control
  and Signal Processing 1986, Elsevier, (1987) 1-5.


  \bibitem{zhang1997quaternions}
F.~Zhang, {\em Quaternions and matrices of quaternions}, Linear algebra
  and its applications, 251 (1997) 21-57.


\bibitem{face}
CBCL Face Database, \emph{MIT Center For Biological and Computation Leearning}, http://www.ai.mit.edu/projects/cbcl.


\end{thebibliography}
\end{document}